\documentclass[reqno,a4paper]{amsart}
\usepackage{amssymb}
\usepackage{amsmath}
\begin{document} 
\newtheorem{prop}{Proposition}[section]
\newtheorem{Def}{Definition}[section] \newtheorem{theorem}{Theorem}[section]
\newtheorem{lemma}{Lemma}[section] \newtheorem{Cor}{Corollary}[section]

\title[LWP for Chern-Simons-Dirac]{\bf A remark on low regularity solutions of the Chern-Simons-Dirac system}
\author[Hartmut Pecher]{
{\bf Hartmut Pecher}\\
Fachbereich Mathematik und Naturwissenschaften\\
Bergische Universit\"at Wuppertal\\
Gau{\ss}str.  20\\
42097 Wuppertal\\
Germany\\
e-mail {\tt pecher@math.uni-wuppertal.de}}
\date{}

\begin{abstract}
An alternative proof of low regularity well-posedness for the Chern-Simons-Dirac system in Coulomb gauge is given which completely avoids the use of any null structure similarly to a recent result of Bournaveas-Candy-Machihara.  An unconditional uniqueness result is also given.
\end{abstract}
\maketitle
\renewcommand{\thefootnote}{\fnsymbol{footnote}}
\footnotetext{\hspace{-1.5em}{\it 2000 Mathematics Subject Classification:} 
35Q40, 35L70 }
\normalsize 
\setcounter{section}{0}
\section{Introduction and main results}
\noindent Consider the Chern-Simons-Dirac system in two space dimensions :
\begin{align}
\label{1.1}
 \frac{1}{2}\epsilon^{\mu \nu \rho} F_{\nu \rho} & =  -J^{\mu} \\
\label{1.2}
i \gamma^{\mu} D_{\mu} \psi = m\psi
 \, ,
\end{align}
with initial data
\begin{equation}
\label{*3}
A_{\mu}(0) = a_{\mu} \quad , \quad \psi(0) = \psi_0  \, , 
\end{equation}
where we use the convention that repeated upper and lower indices are summed, Greek indices run over 0,1,2 and Latin indices over 1,2. Here 
\begin{align*}
D^{\mu}  & := \partial_{\mu} - iA_{\mu} \\
 F_{\mu \nu} & := \partial_{\mu} A_{\nu} - \partial_{\nu} A_{\mu} \\
 J^{\mu} & := \langle \gamma^0 \gamma^{\mu} \psi,\psi \rangle 
\end{align*}
Here $F_{\mu \nu} : {\mathbb R}^{1+2} \to {\mathbb R}$ denotes the curvature, $\psi : {\mathbb R}^{1+2} \to {\mathbb C}^2$ , and $A_{\mu} : {\mathbb R}^{1+2} \to {\mathbb R}$ the gauge potentials. We use the notation $\langle \cdot, \cdot \rangle$ for the inner product in ${\mathbb C}^2$ , $\partial_{\mu} = \frac{\partial}{\partial x_{\mu}}$, where we write $(x^0,x^1,...,x^n) = (t,x^1,...,x^n)$ and also $\partial_0 = \partial_t$ and $\nabla = (\partial_1,\partial_2)$. $\epsilon^{\mu \nu \rho}$ is the totally skew-symmetric tensor with $\epsilon^{012}=1$,  and $m \ge 0$ .
$\gamma^{\mu}$ are the Pauli matrices
$ \gamma^0 = \left( \begin{array}{cc}
1 & 0  \\ 
0 & -1  \end{array} \right)$ 
 , $ \gamma^1 = \left( \begin{array}{cc}
0 & i  \\
i & 0  \end{array} \right)$ , $ \gamma^2 = \left( \begin{array}{cc}
0 & 1  \\
-1 & 0  \end{array} \right) \, .$ 

The equations are invariant under the gauge transformations
$$ A_{\mu} \rightarrow A'_{\mu} = A_{\mu} + \partial_{\mu} \chi \, , \, \psi \rightarrow \psi' = e^{i\chi} \phi \, , \, D_{\mu} \rightarrow D'_{\mu} = \partial_{\mu}-iA'_{\mu} \, . $$
The most common gauges are the Coulomb gauge $\partial^j A_j =0$ , the Lorenz gauge $\partial^{\mu} A_{\mu} = 0$ and the temporal gauge $A_0 = 0$. 

Our main aim is to give a simple proof of local well-posedness for data $\psi_0 \in H^s({\mathbb R}^2)$ for $ s > \frac{1}{4} $, especially we want to show that no null condition is necessary in the case of the Coulomb gauge. Critical with respect to scaling is the case $s=0$. The same result was proven recently in Coulomb as well as Lorenz gauge by M. Okamoto \cite{O} using a null structure of the system. Earlier results were given by Huh \cite{H} in the Lorenz gauge for data $\psi_0 \in H^{\frac{5}{8}} $ , $a_{\mu} \in H^{\frac{1}{2}}$ using a null structure, in the Coulomb gauge for $\psi_0 \in H^{\frac{1}{2}+\epsilon}$ , $ a_i \in L^2$, and in temporal gauge for $\psi_0 \in H^{\frac{3}{4}+\epsilon},$  $a_j \in H^{\frac{3}{4}+\epsilon} + L^2$ , both without using a null structure. Huh-Oh \cite{HO} proved local well-posedness in Lorenz gauge for $\psi_0 \in H^s$ , $a_{\mu} \in H^s$ for $s > \frac{1}{4}$ also making use of a null structure. The methods of Okamoto and Huh-Oh are different. Okamoto reduces the problem to a single Dirac equation with cubic nonlinearity for $\psi$, which does not contain $A_{\mu}$ any longer. From a solution $\psi$ of this equation the potentials $A_{\mu}$ can be constructed by solving a wave equation in Lorenz gauge and an elliptic equation in Coulomb gauge. Our proof also relies on this approach. Huh-Oh on the other hand directly solve a coupled system of a Dirac equation for $\psi$ and a wave equation for $A_{\mu}$.

Our result shows local well-posedness for data $\psi_0 \in H^s$ down to $s=\frac{1}{4}+$ whichout use of the null structure. Therefore it works also for more general systems which lead to cubic Dirac equations of the symbolic form
$$ (-i \alpha^{\mu} \partial_{\mu} +m\beta)\psi \sim \nabla^{-1}\langle \psi,\psi \rangle \psi \, . $$
An almost identical result was recently given by Bournaveas-Candy-Machihara  \cite{BCM} who were also able to avoid any use of the null structure of the system. Their proof relies on a bilinear Strichartz estimate given by Klainerman-Tataru \cite{KT} whereas we make use of bilinear estimates in wave-Sobolev spaces given by d'Ancona-Foschi-Selberg \cite{AFS}.

Our result gives uniqueness in a certain subspace of $C^0([0,T],H^s)$ of $X^{s,b}$-type. Thus it is natural to consider the question whether unconditional uniqueness also holds, namely in $C^0([0,T],H^s)$. We give a positive answer if $s >\frac{1}{3}$ using an idea of Zhou \cite{Z}.

We exclusively study the Coulomb gauge condition $\partial_j A^j = 0$. In this case one easily checks using (\ref{1.1}) that the potentials $A_{\mu}$ satisfy the elliptic equations
\begin{equation}
\label{1.10}
A_0 = \Delta^{-1}(\partial_2 J_1 - \partial_1 J_2) \, , \, A_1 = \Delta^{-1} \partial_2 J_0 \, , \, A_2 = - \Delta^{-1} \partial_1 J_0 \, . 
\end{equation}
Inserting this into (\ref{1.2}) and defining the matrices $\alpha^{\mu} = \gamma^0 \gamma^{\mu}$ , $\beta = \gamma_0$ we obtain
\begin{equation}
\label{1.11}
(i \alpha^{\mu} \partial_{\mu} - m \beta) \psi = N(\psi,\psi,\psi) \, ,
\end{equation}
where
\begin{align*}
 &N(\psi_1,\psi_2,\psi_3) \\
 &= \Delta^{-1}\left( \partial_2 \langle \alpha_1 \psi_1,\psi_2 \rangle - \partial_1 \langle \alpha_2 \psi_1,\psi_2 \rangle + \partial_2\langle \psi_1,\psi_2 \rangle \alpha_1 - \partial_1 \langle \psi_1,\psi_2 \rangle \alpha_2 \right) \psi_3 \, . 
\end{align*}
In the sequel we consider this nonlinear Dirac equation with initial condition
\begin{equation}
\label{1.12}
\psi(0) = \psi_0 \, .
\end{equation}

Using an idea of d'Ancona - Foschi -Selberg \cite{AFS1} we simplify  (\ref{1.11}) by 
considering the projections onto the one-dimensional eigenspaces of the 
operator 
$-i \alpha \cdot \nabla = -i \alpha^j \partial_j$ belonging to the eigenvalues $ \pm |\xi|$. These 
projections are given by $\Pi_{\pm} =$ $\Pi_{\pm}(D)$, where  $ D = 
\frac{\nabla}{i} $ and $\Pi_{\pm}(\xi) = \frac{1}{2}(I 
\pm \frac{\xi}{|\xi|} \cdot \alpha) $. Then $ 
-i\alpha \cdot \nabla = |D| \Pi_+(D) - |D| \Pi_-(D) $ and $ \Pi_{\pm}(\xi) \beta
= \beta \Pi_{\mp}(\xi) $. Defining $ \psi_{\pm} := \Pi_{\pm}(D) \psi$  , the Dirac  equation can be rewritten as
\begin{equation}
\label{4}
(-i \partial_t \pm |D|)\psi_{\pm}  =  m\beta \psi_{\mp} + \Pi_{\pm}N(\psi_+ + \psi_-,\psi_+ + \psi_-, \psi_+ + \psi_-) \, . 
\end{equation}
The initial condition is transformed into
\begin{equation}
\label{6}
\psi_{\pm}(0) = \Pi_{\pm}\psi_0 \, .
\end{equation}

We use the following function spaces and notation. $H^s_p$ denotes the standard $L^p$-based Sobolev space of oder $s$ , $H^s = H^s_2$ , and $B^s_{p,q}$ the Besov space as defined e.g. in \cite{BL}. Let $\, \widehat{} \,$ denote the Fourier 
transform with respect to space and time.
The standard spaces of Bougain-Klainerman-Machedon type $X^{s,b}_{\pm}$ belonging 
to 
the half waves are defined by the completion of ${\mathcal S}({\mathbb R} \times {\mathbb 
R^2})$ with respect to
$$ \|f\|_{X^{s,b}_{\pm}}  = \| \langle \xi 
\rangle^s \langle \tau \pm |\xi| \rangle^b \widehat{f}(\tau,\xi) \|_{L^2} \, ,$$
where $\langle \cdot \rangle := (1+|\cdot|^2)^{\frac{1}{2}}.$ 
$X^{s,b}_{\pm}[0,T]$ is the space of restrictions to the time interval $[0,T]$. 
Similarly $H^{s,b}$ denotes the completion of ${\mathcal S}({\mathbb R} \times {\mathbb 
R^2})$ with respect to
$$ \|f\|_{H^{s,b}}  = \| \langle \xi 
\rangle^s \langle| \tau| - |\xi| \rangle^b \widehat{f}(\tau,\xi) \|_{L^2} $$
and $H^{s,b}[0,T]$ its restriction to the time interval $[0,T]$.

We remark the embedding $X^{s,b}_{\pm} \subset H^{s,b}$ for $b \ge 0$.

Finally $a+$ and $a-$ denote numbers which are slightly larger and smaller than $a$ respectively, such that $a--< a- < a < a+ < a++$ .

We now formulate our results.

\begin{theorem}
\label{Theorem 1.1}
Assume $\psi_0 \in H^s({\mathbb R}^2)$ with $s > \frac{1}{4}$. Then (\ref{1.11}),(\ref{1.12}) is locally well-posed in $H^s({\mathbb R}^2)$. More presicely there are $T>0$ , $b> \frac {1}{2}$ such that there exists a unique solution $\psi = \psi_+ + \psi_-$ with $\psi_{\pm} \in X^{s,b}_{\pm}[0,T]$. This solution belongs to $C^0([0,T],H^s({\mathbb R}^2))$.
\end{theorem}

The unconditional uniqueness result is the following
\begin{theorem}
\label{Theorem 1.2}
Assume $ \psi_0 \in H^s({\mathbb R}^2)$ with $ s > \frac{1}{3}$. The solution of (\ref{1.11}),(\ref{1.12}) is unique in $C^0([0,T],H^s({\mathbb R}^2))$.
\end{theorem}
Fundamental for their proof are the following bilinear estimates in wave-Sobolev spaces which were proven by d'Ancona, Foschi and Selberg in the two dimensional case $n=2$ in \cite{AFS} in a more general form which include many limit cases which we do not need.
\begin{theorem}
\label{Theorem 3.2}
Let $n=2$. The estimate
$$\|uv\|_{H^{-s_0,-b_0}} \lesssim \|u\|_{H^{s_1,b_1}} \|v\|_{H^{s_2,b_2}} $$
holds, provided the following conditions hold:
\begin{align*}
& b_0 + b_1 + b_2 > \frac{1}{2} \\
& b_0 + b_1 > 0 \\
& b_0 + b_2 > 0 \\
& b_1 + b_2 > 0 \\
&s_0+s_1+s_2 > \frac{3}{2} -(b_0+b_1+b_2) 
\\
&s_0+s_1+s_2 > 1 -(b_0+b_1) \\
&s_0+s_1+s_2 > 1 -(b_0+b_2) \\
 &s_0+s_1+s_2 > 1 -(b_1+b_2) \\
 &s_0+s_1+s_2 > \frac{1}{2} - b_0 \\
&s_0+s_1+s_2 > \frac{1}{2} - b_1 \\
&s_0+s_1+s_2 > \frac{1}{2} - b_2 \\
&s_0+s_1+s_2 > \frac{3}{4} 
\end{align*}
 \begin{align*}
&(s_0 + b_0) +2s_1 + 2s_2 > 1 \\
&2s_0+(s_1+b_1)+2s_2 > 1 \\
&2s_0+2s_1+(s_2+b_2) > 1  \\
&s_1 + s_2 \ge \max(0,-b_0) \\
&s_0 + s_2 > \max(0,-b_1) \\
&s_0 + s_1 > \max(0,-b_2) 
\end{align*}
\end{theorem}

\section{Proof of the theorems}
\begin{proof}[Proof of Theorem \ref{Theorem 1.1}]
By standard arguments we only have to show
$$ \| N(\psi_1,\psi_2,\psi_3) \|_{X^{s,-\frac{1}{2}++}_{\pm_4}} \lesssim \prod_{i=1}^3 \|\psi_i\|_{X^{s,\frac{1}{2}+}_{\pm_ i}} \, , $$
where $\pm_i$ $(i=1,2,3,4)$ denote independent signs.

By duality this is reduced to the estimates
$$ J:= \int \langle N(\psi_1,\psi_2,\psi_3),  \psi_4 \rangle dt\, dx \lesssim \prod_{i=1}^3 \|\psi_i\|_{X^{s,\frac{1}{2}+}_{\pm_i}} \|\psi_4\|_{X^{-s,\frac{1}{2}--}_{\pm_4}} \, . $$
By Fourier-Plancherel we obtain
$$ J = \int_* q(\xi_1,...,\xi_4) \prod_{j=1}^4 \widehat{\psi}_j(\xi_j,\tau_j) d\xi_1\, d\tau_1 ... d\xi_4\,d\tau_4 \, , $$
where * denotes integration over $\xi_1-\xi_2=\xi_4-\xi_3=:\xi_0$ and $\tau_1-\tau_2=\tau_4-\tau_3$ and
\begin{align*}
q =  \frac{1}{|\xi_0|^2} &[( \xi_{0_2}(\langle \alpha_1 \widehat{\psi}_1,\widehat{\psi}_2 \rangle \langle \widehat{\psi}_3,\widehat{\psi}_4\rangle +  \langle  \widehat{\psi}_1,\widehat{\psi}_2 \rangle \langle \alpha_1 \widehat{\psi}_3,\widehat{\psi}_4\rangle) \\
& - \xi_{0_1}(\langle \alpha_2 \widehat{\psi}_1,\widehat{\psi}_2 \rangle \langle \widehat{\psi}_3,\widehat{\psi}_4\rangle +  \langle  \widehat{\psi}_1,\widehat{\psi}_2 \rangle \langle \alpha_2 \widehat{\psi}_3,\widehat{\psi}_4\rangle) ] \, .
\end{align*}
The specific structure of this term, namely the form of the matrices $\alpha_j$ plays no role in the following, thus the null structure is completely ignored.

We first consider the case $|\xi_0| \le 1$. In this case we estimate $J$ as follows:
\begin{align*}
 &\| \langle \nabla \rangle^{-s-1} |\nabla|^{-\frac{1}{2}} \langle \alpha_i \psi_1,\psi_2 \rangle \|_{L^2_{xt}} \lesssim
\| \langle \nabla \rangle^{-s-1} \langle \alpha_i \psi_1,\psi_2 \rangle \|_{L^2_{x} L^{\frac{4}{3}}_x} 
\lesssim \|\langle \alpha_i \psi_1,\psi_2 \rangle \|_{L^2_t  B^{-s-1}_{\frac{4}{3},1}} \\
&\lesssim \|\langle \alpha_i \psi_1,\psi_2 \rangle \|_{L^2_t  B^{-s-1+}_{\frac{4}{3},\infty}}
\lesssim \|\langle \alpha_i \psi_1,\psi_2 \rangle \|_{L^2_t  B^{-s}_{1,\infty}}
\lesssim \|\psi_1\|_{L^4_t H^s_x} \|\psi_2\|_{L^4_t H^{-s}_x} \, ,
\end{align*}
where we used the embeddings $B^{-s}_{1,\infty} \subset B^{-s-1+}_{\frac{4}{3},\infty} \subset B^{-s-1}_{\frac{4}{3},1} \subset H^{-s-1}_{\frac{4}{3}}$ , which hold by \cite{BL}, Thm. 6.2.4
and Thm. 6.5.1. The last inequality follows from \cite{RS}, namely the Lemma in Chapter 4.4.3.
The same estimate holds for $\alpha_i = I$. Similarly we obtain 
\begin{align*}
 &\| \langle \nabla \rangle^{-s-1} |\nabla|^{-\frac{1}{2}} \langle \alpha_i \psi_3,\psi_4 \rangle \|_{L^2_{xt}}  
\lesssim \| \psi_3\|_{L^4_t H^s_x} \| \psi_4\|_{L^4_t H^{-s}_x} 
\end{align*}
for arbitrary matrices $\alpha_i$ ,
so that we obtain
\begin{align*}
J &  \lesssim \|\psi_1\|_{X^{s,\frac{1}{4}}_{\pm_1}} \|\psi_2\|_{X^{-s,\frac{1}{4}}_{\pm_2}} \|\psi_3\|_{X^{s,\frac{1}{4}}_{\pm_3}} \|\psi_4\|_{X^{-s,\frac{1}{4}}_{\pm_4}} \, , \end{align*}
which is more than enough.
From now on we assume $|\xi_0| \ge 1$. Assume first that $\frac{3}{4} > s > \frac{1}{4}$. We obtain
\begin{align*}
|J|  \lesssim \sum_{j=1}^2 & \big( \|\langle \alpha_j \psi_1,\psi_2 \rangle \|_{H^{-\frac{1}{4}+,\frac{1}{4}+}} \|\langle \psi_3,\psi_4\rangle\|_{H^{-\frac{3}{4}-,-\frac{1}{4}-}} \\
& + \|\langle \psi_1,\psi_2 \rangle \|_{H^{-\frac{1}{4}+,\frac{1}{4}+}} \|\langle \alpha_j \psi_3,\psi_4\rangle\|_{H^{-\frac{3}{4}-,-\frac{1}{4}-}} \big)\, .
\end{align*}
 By Theorem \ref{Theorem 3.2} with $s_0=\frac{1}{4}-$ , $b_0=-\frac{1}{4}-$ , $s_1=s_2 =s$ , $b_1=b_2=\frac{1}{2}+\epsilon$ for the first factors and $s_0 = \frac{3}{4}+$ , $b_0 = \frac{1}{4}+$ , $s_1=s$ , $s_2=-s$ , $b_1=\frac{1}{2}+\epsilon$ , $b_2=\frac{1}{2}-2\epsilon$ for the second factors we obtain under the assumption $\frac{3}{4} > s > \frac{1}{4}$ :
$$ |J| \lesssim \prod_{j=1}^3 \|\psi_j\|_{H^{s,\frac{1}{2}+\epsilon}} \|\psi_4\|_{H^{-s,\frac{1}{2}-2\epsilon}} \, .$$
Using the embedding $X^{s,b}_{\pm} \subset H^{s,b}$ for $s \in {\mathbb R}$ and $b \ge 0$ we obtain the desired estimate.\\
Next assume $s \ge \frac{3}{4}$. 
We obtain
\begin{align*}
|J|  \lesssim \sum_{j=1}^2 & \big( \|\langle \alpha_j \psi_1,\psi_2 \rangle \|_{H^{s-1+,\frac{1}{4}+}} \|\langle \psi_3,\psi_4\rangle\|_{H^{-s-,-\frac{1}{4}-}} \\
& + \|\langle \psi_1,\psi_2 \rangle \|_{H^{s-1+,\frac{1}{4}+}} \|\langle \alpha_j \psi_3,\psi_4\rangle\|_{H^{-s-,-\frac{1}{4}-}} \big)\, .
\end{align*}
 By Theorem \ref{Theorem 3.2} with $s_0=1-s-$ , $b_0=-\frac{1}{4}-$ , $s_1=s_2 =s$ , $b_1=b_2=\frac{1}{2}+\epsilon$ for the first factors and $s_0 = s+$ , $b_0 = \frac{1}{4}+$ , $s_1=s$ , $s_2=-s$ , $b_1=\frac{1}{2}+\epsilon$ , $b_2=\frac{1}{2}-2\epsilon$ for the second factors we obtain the same estimate as before.
\end{proof}
{\bf Remark:} The potentials are completely determined by $\psi$ and (\ref{1.10}). We have $A_{\mu} \sim |\nabla|^{-1} \langle \psi, \psi \rangle$ , so that for $s<1$ :
$$ \|A_{\mu}\|_{\dot{H}^{2s}} \lesssim \|\langle \psi,\psi\rangle \|_{\dot{H}^{2s-1}} \lesssim  \|\langle \psi,\psi\rangle \|_{L^{\frac{1}{1-s}}} \lesssim \|\psi\|^2_{L^{\frac{2}{1-s}}} \lesssim \|\psi\|^2_{H^s} < \infty $$
and
$$ \|A_{\mu}\|_{\dot{H}^{\epsilon}} \lesssim \|\langle \psi,\psi\rangle \|_{\dot{H}^{\epsilon -1}} \lesssim \|\psi\|^2_{L^{\frac{4}{2-\epsilon}}} \lesssim \|\psi\|^2_{H^s} < \infty \, ,$$
thus we obtain for $0<\epsilon\ll 1$ and $s<1$ :
$$ A_{\mu} \in C^0([0,T],\dot{H}^{2s} \cap \dot{H}^{\epsilon}) \,. $$

\begin{proof}[Proof of Theorem \ref{Theorem 1.2}]
We first show $\psi_{\pm} \in X^{0,1}_{\pm}[0,T]$. We have to prove
$$ \|N(\psi_1,\psi_2,\psi_3)\|_{L^2_t([0,T],L^2_x)} \lesssim \prod_{j=1}^3 \|\psi_j\|_{L^{\infty}_t([0,T], H^{\frac{1}{3}}_x)} \, , $$
where the implicit constant may depend on $T$ .
This follows from the estimate
\begin{align*}
\| |\nabla|^{-1} \langle \alpha_i \psi_j,\psi_k \rangle \psi_3\|_{L^2_x} & \lesssim \| |\nabla|^{-1} \langle \alpha_i \psi_j,\psi_k \rangle \|_{L^6_x} \|\psi_3\|_{L^3_x} \lesssim \| \langle \alpha_i \psi_j,\psi_k \rangle \|_{L^{\frac{3}{2}}_x} \|\psi_3\|_{L^3_x} \\
& \lesssim  \|\psi_j\|_{L^3_x} \|\psi_k\|_{L^3_x} \|\psi_3\|_{L^3_x} \lesssim \|\psi_j\|_{H^{\frac{1}{3}}_x} \|\psi_k\|_{H^{\frac{1}{3}}_x}\|\psi_3\|_{H^{\frac{1}{3}}_x} \, ,
\end{align*}
 and a similar estimate for the term $\| |\nabla|^{-1} \langle  \psi_j,\psi_k \rangle \alpha_i \psi_3\|_{L^2_x}$ .\\
Assume now $\psi \in C^0([0,T],H^{\frac{1}{3}+\epsilon})$ , $\epsilon > 0$. Then we have shown that $\psi_{\pm} \in X_{\pm}^{\frac{1}{3}+\epsilon,0}[0,T] \cap X^{0,1}_{\pm}[0,T]$. By interpolation we get $\psi_{\pm} \in X^{\frac{1}{4}+\frac{\epsilon}{4},\frac{1}{4}+\epsilon}_{\pm} [0,T] $ for $ \epsilon \ll 1.$ \\
Assume now that $\psi,\psi' \in C^0([0,T],H^{\frac{1}{3}+\epsilon})$ are two solutions of (\ref{1.11}),(\ref{1.12}), Then we have
\begin{align}
\label{**}
\sum_{\pm} \|\psi_{\pm}-\psi_{\pm}'\|_{X^{0,\frac{1}{2}+}_{\pm}[0,T]} 
&\lesssim T^{0+} \sum_{\pm} \|N(\psi,\psi,\psi)-N(\psi',\psi',\psi')\|_{X_{\pm}^{0,-\frac{1}{2}++}[0,T]}
\\ \nonumber
& \lesssim T^{0+} \sum_{\pm,\pm_1,\pm_2,\pm_3} \big(\|N(\psi_{\pm_1}-\psi_{\pm_1}',\psi_{\pm_2},\psi_{\pm_3})\|_{X_{\pm}^{0,-\frac{1}{2}++}[0,T]}\\ \nonumber
& \hspace{6.5em}+  \|N(\psi_{\pm_1}',\psi_{\pm_2}-\psi_{\pm_2}',\psi_{\pm_3})\|_{X_{\pm}^{0,-\frac{1}{2}++}[0,T]}\\ \nonumber
&\hspace{6.5em} +\|N(\psi_{\pm_1}',\psi_{\pm_2}',\psi_{\pm_3}-\psi_{\pm_3}')\|_{X_{\pm}^{0,-\frac{1}{2}++}[0,T]}\big)
\end{align}
Here $\pm$,$\pm_j$ $(j=1,2,3)$ denote independent signs. We want to show that for the first term the following estimate holds:
\begin{align}
\label{*}
\nonumber
J:= &\int \langle N(\psi_{\pm_1}-\psi_{\pm_1}',\psi_{\pm_2},\psi_{\pm_3}),\psi_{4} \rangle dx\,dt \\ 
&\lesssim \|\psi_{\pm_1}-\psi_{\pm_1}'\|_{X^{0,\frac{1}{2}+}_{\pm_1}} \|\psi_{\pm_2}\|_{X^{\frac{1}{4}+\frac{\epsilon}{4},\frac{1}{4}+\epsilon}_{\pm_2}} 
\|\psi_{\pm_3}\|_{X^{\frac{1}{4}+\frac{\epsilon}{4},\frac{1}{4}+\epsilon}_{\pm_3}}
\|\psi_{4}\|_{X^{0,\frac{1}{2}--}_{\pm_4}} \, .
\end{align}
We consider the case $|\xi_0|\le 1$ first. Similarly as in the proof of Theorem \ref{Theorem 1.1} we obtain
$$ |J| \lesssim \|\psi_{\pm_1} - \psi_{\pm_1}'\|_{X^{-\frac{1}{4}-\frac{\epsilon}{4},\frac{1}{4}}_{\pm_1}}  \|\psi_{\pm_2}\|_{X^{\frac{1}{4}+\frac{\epsilon}{4},\frac{1}{4}}_{\pm_2}}
\|\psi_{\pm_3}\|_{X^{\frac{1}{4}+\frac{\epsilon}{4},\frac{1}{4}}_{\pm_3}} 
\|\psi_{4}\|_{X^{-\frac{1}{4}-\frac{\epsilon}{4},\frac{1}{4}}_{\pm_4}}\, , $$
which is more than sufficient. For $|\xi_0| \ge 1$ we obtain
\begin{align*}
|J| &\lesssim \sum_{j=1}^2 \big( \| \langle \alpha_j(\psi_{\pm_1} - \psi_{\pm_1}'),\psi_{\pm_2} \rangle \|_{H^{-\frac{1}{2},0}} \|\langle \psi_{\pm_3},\psi_4 \rangle\|_{H^{-\frac{1}{2},0}}\\ &\hspace{2.5em}
 + \| \langle \psi_{\pm_1} - \psi_{\pm_1}',\psi_{\pm_2} \rangle \|_{H^{-\frac{1}{2},0}} \|\langle \alpha_j \psi_{\pm_3},\psi_4 \rangle\|_{H^{-\frac{1}{2},0}}\big)\\
 & \lesssim \|\psi_{\pm_1} - \psi_{\pm_1}'\|_{H^{0,\frac{1}{2}+}} \|\psi_{\pm_2}\|_{H^{\frac{1}{4}+\frac{\epsilon}{4},\frac{1}{4}+\epsilon}}
 \|\psi_{\pm_3}\|_{H^{\frac{1}{4}+\frac{\epsilon}{4},\frac{1}{4}+\epsilon}}
 \|\psi_4\|_{H^{0,\frac{1}{2}--}} \, ,
\end{align*}
where we used Theorem \ref{Theorem 3.2} for the first factor with the choice
$s_0 = \frac{1}{2}$ , $b_0=0$ , $s_1=0$ , $b_1=\frac{1}{2}+$ , $s_2=\frac{1}{4}+\frac{\epsilon}{4}$ , $b_2 = \frac{1}{4}+\epsilon$ and for the second factor with $s_0 = \frac{1}{2}$ , $b_0=0$ , $s_1= \frac{1}{4}+\frac{\epsilon}{4}$ , $b_1 = \frac{1}{4}+\epsilon$ , $s_2=0$ , $b_2= \frac{1}{2}--$. The embedding  $X^{s,b}_{\pm} \subset H^{s,b}$ for $b \ge 0$ gives (\ref{*}). The other terms in (\ref{**}) are treated similarly. We obtain
\begin{align*}
&\sum_{\pm} \|\psi_{\pm} - \psi_{\pm}'\|_{X^{0,\frac{1}{2}+}_{\pm}[0,T]} \\
& \lesssim T^{0+} \sum_{j=1}^2 \big( \|\psi_{\pm_j}\|^2_{X^{\frac{1}{4}+\frac{\epsilon}{4}, \frac{1}{4}+\epsilon}_{\pm_j}[0,T]} +\|\psi_{\pm_j}'\|^2_{X^{\frac{1}{4}+\frac{\epsilon}{4}, \frac{1}{4}+\epsilon}_{\pm_j}[0,T]} \big) \sum_{\pm} \|\psi_{\pm}-\psi_{\pm}'\|_{X^{0,\frac{1}{2}+}_{\pm}[0,T]} \, .
\end{align*}
For sufficiently small $T$ this implies $\|\psi_{\pm} - \psi_{\pm}'\|_{X^{0,\frac{1}{2}+}_{\pm}[0,T]}=0$ , thus local uniqueness.
By iteration $T$ can be chosen arbitrarily.
\end{proof}

\end{document}